\documentclass[a4paper,10pt]{article}
\usepackage{amscd}
\usepackage{amsmath}
\usepackage{bbding}
\usepackage{amsfonts}
\usepackage{cite}
\usepackage{mathrsfs}
\usepackage{stmaryrd}
\usepackage{graphicx,latexsym,amsmath,amssymb,amsthm,enumerate}
\usepackage{bm}
\usepackage{color}

\theoremstyle{definition}
\newtheorem{Theorem}{Theorem}[section]
\newtheorem{Lemma}{Lemma}[section]

\newtheorem{Proposition}{Proposition}[section]
\newtheorem{Definition}{Definition}[section]

\newtheorem{Remark}{Remark}[section]

\def\acknowledgement{\par\addvspace{17pt}\small\rmfamily
	\trivlist\if!\ackname!\item[]\else
	\item[\hskip\labelsep
	{\bfseries\ackname}]\fi}

\def\ackname{Acknowledgements}
\hfuzz=\maxdimen
\begin{document}
	\title{{\LARGE\bf{Image Space Analysis to the generalized optimization problem on a local sphere}}}
	\author{Li-wen Zhou$^{a,b}$\footnote{E-mail address: zhouliwen@live.cn} , Min Tang $^{a}$, Ya-ling Yi $^{a}$ and Yao-Jia Zhang$^{a,b}$\\
		$^a${\small\it School of Sciences, Southwest Petroleum University, Chengdu, 610500, P.R. China}\\
		$^b${\small\it Institute for Artificial Intelligence, Southwest Petroleum University, Chengdu 610500, Sichuan, P. R. China}}
	\date{ }
	\maketitle
	\vspace*{-9mm}
	\begin{center}
		\begin{minipage}{5.6in}
			\noindent{\bf Abstract.} This paper introduces and studies the generalized optimization problem (for short, GOP) defined by the conic order relation on a local sphere. The existence of solution to this problem is studied by using image space analysis (for short, ISA), and a class of regular weak separation functions on the local sphere is established. Moreover, a Lagrangian-type sufficient optimality condition and a saddle-point-type necessary optimality condition for GOP is obtained by a second class of weak separation functions, which are based on the Gerstewitz function and the directional distance function. The problem is transformed into a solvable real-valued optimization problem using scalarization methods.
		\end{minipage}
	\end{center}
	
	\section{Introduction} \noindent
	\setcounter{equation}{0}
	
In recent years, several important concepts from nonlinear analysis and optimization have been extended from Euclidean spaces to spherical settings, and it is natural to generalize the principles and methods of optimization from Euclidean spaces to the Euclidean sphere, enabling further progress in the investigation of convex theory, fixed point theory, variational inequalities, and related topics (see \cite{KW,UCC}). The motivations for this extension may be rooted in theoretical considerations or focused on the development of efficient algorithms (see \cite{Dsp,Drz,Sst}). In general, the sphere is a non-linear space where geodesics are parts of the longest arc, in contrast to Euclidean space, where geodesics are straight lines and the shortest path between any two points is always a straight line segment. However, these concepts and techniques can be extended from linear spaces to local spheres. In fact, over the past few decades, the concepts and techniques applicable to Euclidean space have expanded to the non-linear framework of the sphere (see \cite{KW,ZLW,LXB}). 
	
In particular, as we know, the extension of GOPs from Euclidean spaces to manifolds has some important advantages. For example, nonconvex problems can be translated into convex problems with appropriate Riemannian metrics. Additionally, from the perspective of Riemannian geometry, constrained problems may become unconstrained problems (see \cite{tr,UCC,jka}). This study is valuable because the sphere possesses a specific underlying algebraic structure that can be exploited to significantly reduce the cost of obtaining solutions. The key to addressing this problem lies in establishing an appropriate order relation through the use of cones. Therefore, proposing  order relation using cones in the tangent space of the sphere is essential for studying GOPs on the sphere. 
	
On the other hand, in the last decades, the ISA for constrained extremum problems has been of great interest in the academic and professional communities. Giannessi initially investigated vector optimization problems (VOPs) and vector variational inequalities (VVIs) starting from a general scheme based on a cone in finite-dimensional Euclidean space (see\cite{GF1,GF2,GF3}). Since then, ISA has emerged as a crucial approach and has been widely applied to various classes of optimization problems. Over the past decades, several theoretical aspects of GOPs have been extensively studied, including saddle points, Lagrange functions, optimality conditions, penalty methods, and dualities. This method provides a unified framework for the theoretical exploration of nonconvex, discontinuous, and nonsmooth constrained optimization problems, as well as variational inequalities. ISA can be applied to any infeasibility problem that can be expressed as a parameter system, and this parametric infeasibility can be equivalently transformed into the separation of two nonempty sets $\mathcal{K}$ and $\mathcal{H}$ in the image space, where set $\mathcal{K}$ consists of the images of the objective function and the constraint function, while set $\mathcal{H}$ determined by the nature of the constraints and is a convex cone. Moreover, the disjunction between two sets can be established by the level sets of a separation function. Thus, it is important to introduce a class of suitable separation functions in this problem. Currently, many scholars have constructed and proved various forms of separation functions. Li developed nonlinear weak and strong separation functions for constrained extremum problems by using directed distance functions (see \cite{LJH}). For the vector likelihood-variable inequality problem, Chen proposed four new classes of weak separation functions by applying indicator functions alongside oriented distance functions (see\cite{cjw}). Additionally, a class of nonlinear weak separation functions for constrained vector optimization problems was also proposed and proven by using the Gerstewitz function (see\cite{YMXL}).
	
Inspired by the above work and other references (see \cite{LJH,cjw,YMXL}), we focus on two classes of separation functions and derive the corresponding generalized Lagrange functions, and construct a sufficient condition for Lagrange-type optimality and a necessary condition for saddle-point-type optimality of the GOPs. Besides, we introduce scalarization methods to transform the problem into a solvable real-valued optimization problem.

The rest of this paper is organized as follows. Section 2 presents the necessary concepts, lemmas, and propositions. In Section 3, we establish the GOP by the definition of cone order relation on the local sphere. In Section 4, we give two classes of separation functions to get the existence of solutions to the GOP by the level set. Section 5 introduces the corresponding scalar minimization problem and establishes the foundation for subsequent algorithmic constructions and applications. Conclusions are presented at the end of this paper.
	
\section{Preliminaries}\noindent
	\setcounter{equation}{0}

In this section we recall some notations, definitions and basic properties about the geometry of the sphere used throughout this paper (see \cite{DCM}).

In this paper, unless otherwise stated, the symbols $X$, $Y$, and $Z$ are used to denote normed linear spaces. For simplicity, the norms of these spaces are denoted by $\|\cdot\|$, and the Euclidean inner product is denoted by $\langle \cdot , \cdot \rangle$. For a nonempty set $A\subset Y$, the closure, convex hull, complement, interior, and boundary of $A$ are denoted by $clA, $  $convA, $  $A^c, $ $int A,$  $bd A,$ respectively. 

Let $S^{2}=\{x=(x_1,...,x_{3})\in R^{3}{|}\parallel x\parallel=1\}$ be a sphere, $M$ be a local sphere (see \cite{FOP}). For any $x \in M\subset S^{2}$, the tangent space of $M$ at $x$ denoted by $T_x M$ is a linear subspace, and the tangent bundle of $M$ is defined by
$$TM=\bigcup_{x\in M} T_x M,$$
where $T_xM=\{v\in{R}^2{|}\langle x,v\rangle=0\}$. A vector field $V$ on $M$ is a mapping of $M$ into $TM$, which associates to each point $x\in M$ a vector $V(x) \in T_x M$. 

As considered in \cite{FOP}, the intrinsic distance on the sphere between two arbitrary points $x,y\in M$ is defined by
$$d(x,y)=\arccos{\langle x,y\rangle}.$$

For any given $p\in M$, the intrinsic distance function on the local sphere from $p$ is defined by:
$$d_q(p)=\arccos{\langle p,q\rangle}.$$

As considered in \cite{FOP}, it can be shown that the intrinsic distance $d(x,y)$ between two arbitrary points $x,y\in M$ is obtained by minimizing the arc length functional $L$
$$L(\gamma)=\int_a^b\parallel\gamma^{\prime}(t)\parallel dt.$$
over the set of all piecewise continuously differentiable curves $\gamma\;{:}\;[a,b]\to M$ joining $x$ to $y$, i.e., such that $x=\gamma(a)$ and $y=\gamma(b)$ . Moreover, $d$ is a distance in $M$ and $(M,d)$ is a complete metric space, so that 
\begin{itemize}
	\item[(i)]$d(x,y)\geq0$ for all $\forall x,y\in M$;
	\item[(ii)]$d(x,y)=0$ iff $x=y$; 
	\item[(iii)]$d(x,y)\leq\pi $ for all $\forall x,y\in M$;
	\item[(iv)]$d(x,y)=\pi $ iff $x=-y$.
\end{itemize}	 

As considered in \cite{FOP}, the geodesic is defined as the intersection curve of a plane though the origin of ${R}^3$ with the sphere $M,$ and the geodesic segment $\gamma\;{:}\;[a,b]\to M$ minimal if its arc length is equal the intrinsic distance between its end points, i.e.,  if $L(\gamma)=\operatorname{arccos}\langle\gamma(a),\gamma(b)\rangle $. It is said that $\gamma$ is a normalized geodesic
iff $\|\gamma^{\prime}\|=1$.  If $x\neq y$ and $x\neq-y$, then the unique segment of minimal normalized geodesic from to $x$ to $y$ is denoted by
$$\gamma_{xy}(t)=\left(\cos t-\frac{\langle x,y\rangle\mathrm{sin}t}{\sqrt{1-\langle x,y\rangle^2}}\right)x+\frac{\sin t}{\sqrt{1-\langle x,y\rangle^2}}y,\:\;\; t\in[0,d(x,y)]. $$


A vector field $V$ is said to be parallel along a smooth curve $\gamma\in M$ iff $\nabla_{\gamma^{\prime}}V=0$, where $\nabla$ is the Levi-Civita connection associated with $(M,\langle\cdot,\cdot\rangle)$. We use $\nabla$ to introduce an isometry $P_{\gamma,\cdot,\cdot},$ on the tangent bundle $TM$ along $\gamma$,  which is called the parallel transport and defined as
$$P_{\gamma,\gamma(b),\gamma(a)}(v)=V(\gamma(b)),\;\;\forall a,b\in R \;\; and\;\;  v\in T_{\gamma(a)}M,$$
where $V$ is the unique vector field satisfying $\nabla_{\gamma^{\prime}(t)}V=0$ for all $t$ and $V(\gamma(a))=v$. Without any confusion, $\gamma$ can be omitted when it is a minimal geodesic joining $x$ to $y$.

\begin{Definition}(\cite{FOP})\label{de2.1}
	The exponential mapping $\exp_p\;{:}\;T_pM\to M$ is defined by $\exp_p\nu=\gamma_\nu(1)$, where $\gamma_v$ is the geodesic defined by its initial position $p$,  with velocity $v$ at $p$. Hence,
	$$\exp_pv=\begin{cases}\cos{(\parallel v\parallel)p}+\sin{(\parallel v\parallel)}\frac{v}{\parallel v\parallel},&\quad v\in T_pM/\{0\},\\p,&\quad v=0.\end{cases}$$
	
	We have $\gamma_{t\nu}(1)=\gamma_\nu(t)$ for any values of $t$. Therefore, for all ${t}\in{R}$ we have
	$$\exp_ptv=\begin{cases}\cos{(t\parallel v\parallel)p}+\sin{(t\parallel v\parallel)}\frac{v}{\parallel v\parallel},&\quad v\in T_pM/\{0\},\\p,&\quad v=0.\end{cases}$$
	The inverse of the exponential mapping is given by 
	$$\exp_p^{-1} q=\begin{cases}\frac{\arccos\langle p,q\rangle}{\sqrt{1-\langle p,q\rangle^2}}(I-pp^{\mathrm{T}})q,&q\notin\{p,-p\},\\0,&q=p,\end{cases}$$
	and $\exp_p^{-1}q$ is the velocity $v\in T_pM$ at $p$ of the geodesic from $p$ to $q.$	
	The combination of the aforementioned equations yields the result that $d(p,q)=\parallel\exp_q^{-1}p\parallel$, where $p$ and $q$ are both elements of $M$. This is tantamount to stating that $d_q(p)=\parallel\exp_q^{-1}p\parallel $. The distance $d_{q}$ between the point $q$ and $p$, where $p\in M\setminus\{q,-q\}$, $d_{q}$ is twice differentiable at $p$.
\end{Definition}

\begin{Lemma}(\cite{DCM})\label{le2.1}
	Let $M$ be a local sphere, then the universal cover of $M$ is a convex geodesic space with respect to the induced length metric $d$.
\end{Lemma}

\begin{Lemma}(\cite{DCM})\label{pr2.2}
	Let $M$ be a local sphere and $p\in M$. Then,  $\exp_p:T_pM\to M$ is a diffeomorphism. For any two points $p,q\in M$, if $d(p,q)<\pi $, there exists a unique minimal geodesic $\gamma:[0,1]\to M$ such that $\gamma_{p,q}(0)=p$ and $\gamma_{p,q}(1)=q$. Furthermore, there exists a unique minimal geodesic given by $\gamma_{p,q}(t)=\exp_pt\exp_p^{-1}q$ for all $t\in[0,1]$ joining $p$ to $q$.
\end{Lemma}

\begin{Lemma}(\cite{DCM})\label{pr2.3}
	The exponential map and its inverse are continuous on a local sphere.	
\end{Lemma}

\begin{Lemma}(\cite{LCL})\label{le2.2}
	Let $x_0\in M$ and $\{x_n\}\subset M$ such that $x_n\to x_0$, then the following assertion holds.
	\begin{itemize}
		\item[(i)]For any given $y\in M$, $\exp_{x_n}^{-1}y\to \exp_{x_0}^{-1}y$ and $\exp_y^{-1}x_n\to \exp_y^{-1}x_0$;
		
		\item[(ii)]If $\left\{v_{n}\right\}$ is a sequence such that $v_n\in T_{{x}_n}M$ and $v_{n}\to v_{0}$, then $v_{0}\in T_{x_{0}}M$;
		
		\item[(iii)]Given the sequence $\{u_n\}$ and $\{v_{n}\}$ satisfy $u_n,v_n\in T_{x_n}M$, if $u_{n}\to u_{0}$ and $v_{n}\to v_{0}$ with $u_0,v_0\in T_{x_0}M$, then 
		$$\langle u_n,v_n\rangle\to\langle u_0,v_0\rangle.$$
	\end{itemize}
\end{Lemma}

\begin{Definition}(\cite{UCC})\label{de2.2}
	A subset $C\subset M$ is said to be geodesic convex iff for any two points $x,y\in C$, the geodesic joining $x$ to $y$ is contained in $C$, that is, if $\gamma{:}[a,b]\to M$ is a geodesic such that $x=\gamma(a)$ and $y=\gamma(b)$, then
	$$\gamma((1-t)a+tb)\in C,$$
	for all $t\in[0,1]$.
\end{Definition}

\begin{Remark}
	On the local sphere $M$, a subset $C\subset M$ is geodesic convex iff
	$\exp_x t\exp_x^{-1}y\in C,$
	for all $x,y\in C$ and $t\in[0,1]$.
\end{Remark}

\begin{Definition}(\cite{YMXL})
	For any mapping $f:E\to{R}$ and $\alpha\in{R}$, the set
	$$\operatorname{lev}_{\geq(>)\alpha}f=\{x\in E|f(x)\geq(>)\alpha\} $$
	is called nonnegative (positive) level set of $f$.
\end{Definition}

\begin{Definition}(\cite{YMXL})\label{de2.3}
	Let $p$ be an point on $M$ and $D_p$ be a cone on $T_p M$. The polar cone of $D\subseteq T_pM$ is defined as follow:
	$$D_p^*=\{a\in T_pM\;|\;\langle a,b\rangle\geq0,\;\;\forall b\in D_p\},$$
	where $\langle\cdot,\cdot\rangle$ is the scalar product on the tangent space $T_{p}M$ is the scalar product on the tangent space.
\end{Definition}

\begin{Definition}(\cite{CJ})\label{de2.4}
	For any given set $A\subseteq Y$, the oriented distance function of $A$ is the function $\Delta_A\colon Y\to{R}\cup\{\pm\infty\}$ which is defined by
	$$\Delta_A(y)=d_A(y)-d_{Y\setminus A}(y),$$
	where $d_{A}(y)=\inf_{z\in A}\|y-z\|,\;\;d_{\phi}(y)=+\infty$.
\end{Definition}

\begin{Lemma}(\cite{CJ})\label{pr2.4}
	If the set $A$ is nonempty and $A\neq Y$, then
	\begin{itemize}
		\item[(i)]$\Delta_{A}$ is real valued;
		
		\item[(ii)]$\Delta_{A}$ is 1-Lipschitzian;
		
		\item[(iii)]$\Delta_A(y)<0$ for every $y\in\text{int}A$;
		
		\item[(iv)]$\Delta_A(y)=0$ for every $y\in \text{bd}A$;
		
		\item[(v)]$\Delta_A(y)>0$ for every $y\in \text{int}A^c$;
		
		\item[(vi)]if $A$ is a cone, then $\Delta_{A}$ is positively homogeneous;
		
		\item[(vii)]if $A$ is convex, then $\Delta_{A}$ is convex;
		
		\item[(viii)]if $A$ is a closed convex cone, then $\Delta_{A}$ is nonincreasing with respect to the order relation induced on $Y$, i.e., the following is true: if $y_1,y_2\in Y$, then
		$$y_1-y_2\in A\Rightarrow\Delta_A(y_1)\leq\Delta_A(y_2);$$
		if $A$ has a nonempty interior, then
		$$y_1-y_2\in \text{int}A\Rightarrow\Delta_A(y_1)<\Delta_A(y_2).$$
	\end{itemize}
\end{Lemma}

\begin{Definition}(\cite{GB})\label{de2.5}
	Let $E$ be a nonempty subset of $Y,$ and $K\subseteq E$ be a pointed, closed, and convex cone with nonempty interior. For any given $q\in-\text{int}K$, a mapping $\xi_{K,q}\colon E\to{R}$ denoted by
	$$\xi_{K,q}(y)=\min\{t\in{R} \;{|}\;y\in tq+K\}$$
	is called Gerstewitz function.
\end{Definition}

\begin{Lemma}(\cite{GB})\label{pr2.5}
	For any given $q\in-\text{int}K$, $y\in E$, and $r\in{R}$, the following conclusion holds:
	\begin{itemize}
		\item[(i)]$\xi_{K,q}(y)<r\Leftrightarrow y\in rq+\text{int}K$;
		
		\item[(ii)]$\xi_{K,q}(y)\leq r\Leftrightarrow y\in rq+K$;
		
		\item[(iii)]$\xi_{K.q}(y)=r\Leftrightarrow y\in rq+\mathrm{bd}K$;
		
		\item[(iii)]$\xi_{K,q}(y)$ is convex;
		
		\item[(v)]$\xi_{K,q}(y)$ is Lipschitz continuous;
		
		\item[(vi)]$\xi_{K,q}(y)$ is subadditive.
	\end{itemize}
\end{Lemma}

	\section{A generalized optimization problem}\noindent
\setcounter{equation}{0}

We establish a GOP on a local sphere by a cone order relation on the tangent space. Let $p$ be a given point in $M$ and $C_p$ be a given cone on the tangent space $T_p M$. In the following, we assume that $C_p$ is a closed convex cone and $intC_p\neq\emptyset $. For any $x\in M$, we define 
	$$C_x=P_{x,p}C_p=\{c_x\in T_xM\;|c_x=P_{x,p}c_p,\forall c_p\in C_p\}$$
	by the parallel transport from $p$ to $x$, and let 
	$$\tilde{C}_x=\{\exp_xc_x\;|\;\forall c_x\in C_x\}$$
	be a subset of $M$ by the exponential mapping.
	
	Now we can compare two points $x,y\in M$ with the given cone $C_p$. From now on, $C_x$ is a convex with a given point $p\in M$ and a convex cone $C_p\subset T_pM$.
	
	\begin{Definition}\label{de3.1}
		Let $x$, $y$ be any given points on $M$. It is said that
		$$y <_{\tilde{C}_x} x$$
		iff $\exp_x^{-1}y\in C_x\setminus\{0\}$, and
		$$y \not<_{\tilde{C}_x} x$$
		iff $\exp_x^{-1}y\notin C_x\setminus\{0\}$.
	\end{Definition}
	
	\begin{Remark}
		Due to the special curvature properties of the local sphere, consider two points on the sphere $p$, $q$ and their tangent spaces $T_p M$ and $T_q M$ respectively. For any given vector $v\in C_p$, there exiists a certain geodesic $\gamma$ parallel to $v$ from point $p$ to point $q$, the resulting vector $v^{\prime}$ is obtained which, if it belongs to the tangent space $T_q M$ of of the cone in the tangent space $C_p$, can be compared with the relation between the magnitudes of p and q to establish the order relation.
		Thus, we can observe that the order relation on the local sphere $M$ is equivalent to the order relation $T_x M$ on the local sphere with respect to $C_x$ on the local sphere. The latter is often mentioned in the field of VOPs and VVIs. That is
		$$y<_{\tilde{C}_x}x\Leftrightarrow0<_{C_x\setminus\{0\}}\exp_x^{-1}y\Leftrightarrow \exp_x^{-1}y\in C_x\setminus\{0\},$$
		$$y\not<_{\tilde{C}_x}x\Leftrightarrow0\not<_{C_x\setminus\{0\}}\exp_x^{-1}y\Leftrightarrow \exp_x^{-1}y\notin C_x\setminus\{0\}.$$
	\end{Remark}
	
	Let us consider the function $f:M\to M$ and vector value function $g=(g_1,g_2,...,g_l): M\to R^l$, the locally spherical GOP with respect to the cone $C_p\subset T_pM$ is introduced as follow,
	\begin{equation}\label{question}
		\begin{aligned}&\mbox{min}_{C_p}f(x),\\&\mathrm\;{s.t.}\;x\in K= \{ x \in M \mid g(x) \in R_+^l \},\end{aligned}
	\end{equation}
	where $\min_{C_p}$ is the minimum with respect to the cone $C_p$, and $g(x) \in R_+^l$ means every components $g_i(x)\geq0\;\;(i=1,...,l)$.
	
	Next, we define the notion of an efficient solution to (\ref{question}).
	
	\begin{Definition}\label{de3.2}
		A feasible point $y\in K$ is called an efficient solution of \eqref{question} iff
		\begin{equation}\label{EQ1}
			f(x)\not<_{\tilde{C}_{f(y)}} f(y),\;\;\forall x\in K.
		\end{equation}
		The inequality here is equivalent to $f(y)-f(x)\notin C_{f(y)}\setminus\{0\}$, where
		$$\tilde{C}_{f(y)}=\{\exp_{f(y)}c_{f(y)}\;|\;\forall c_{f(y)}\in C_{f(y)}\},$$
		and 
		$$C_{f(y)}=P_{f(y),p}C_{p},$$
		where $C_{f(y)}$ is a convex cone. 
		Inequality \eqref{EQ1} means
		\begin{equation}\label{EQ3}
			\exp_{f(y)}^{-1}f(x)\notin C_{f(y)}\setminus\{0\}, \;\;\forall x\in K.
		\end{equation}
	\end{Definition}
	
	It is trivial to note that \eqref{EQ1} is satisfied iff for any $x\in M$, $g(x)\in R_+^l$
	\begin{equation}\label{EQ4}
		\exp_{f(y)}^{-1}f(x)\geq_{{C}_{f(y)}}0,
	\end{equation}
	is impossible.
	
	\begin{Remark}
		If $M$ reduces to an Euclidean space $R^n$ and $T_x M=M=R^n$ for all $x\in M$, then problem \eqref{question} becomes a classical VOP.
	\end{Remark}
	
	In order to study the optimality conditions for the solution of the GOP \eqref{question}, by using the ISA, we transform the existence of solution to \eqref{question} into the separation of two suitable subsets in the image space.
	
	From the previous definition, if $y$ is an efficient solution of problem \eqref{question} iff
	\begin{equation}
		\exp_{f(y)}^{-1}f(x)\in C_{f(y)}\setminus\{0\}, \;g(x)\in R_+^l,\;\;\forall x\in M
	\end{equation}
	is impossible.
	
	Consider the following two sets:
	\begin{equation}\label{EQ5}
		\mathcal{H}(y)=\{(u,v)\in T_{f(y)}M\times R^l\;|\;u\in C_{f(y)}\setminus\{0\},v\in R_+^l\},
	\end{equation}
	\begin{equation}\label{EQ6}
		\mathcal{K}(y)=\{(u,v)\in T_{f(y)}M\times{R}^l\;|\;u=\exp_{f(y)}^{-1}f(x),v=g(x),\;\;\forall x\in M\},
	\end{equation}
	In the following discussion, we denote $\mathcal{H}(y)$ and $\mathcal{K}(y)$ simply as $\mathcal{H}$ and $\mathcal{K}$, respectively, when there is no risk of ambiguity. The set $\mathcal{K}(y)$ is referred to as the image of problem \eqref{question}, while the space $T_{f(y)}M\times{R}^l$ is identified as the image space (IS) of problem \eqref{question}. We can reformulate problem \eqref{question} by using the mapping $M_{y}\colon M\to T_{f(y)}M\times R^{l}$, where $M_y(x)=(\exp_{f(y)}^{-1}f(x),g(x))$. This leads to the following optimization problem:
	\begin{equation}
		\begin{aligned}&\mbox{max}_{C_p} u,\\&\mathrm \;{s.t.}(u,v)\in \mathcal{K}\cap(T_{f(y)}M\times R^l),\end{aligned}
	\end{equation}
	where $\max_{C_p}$ denotes the maximum value with respect to the cone $C_p$.
	
	The above two sets of equivalence relations portray the optimality of a feasible point. In general, even if the function involved in a GOP has some convexity, there are usually images that $\mathcal{K}$ is not convex. In order to overcome this difficulty, it is necessary to regularise the image $\mathcal{K}$ regularisation, i.e. cone expansion, is defined as follows:
	\begin{equation}\label{EQ2}
		\varepsilon(y)= \mathcal{K} -\mathrm{cl}\mathcal{H} \\
		= \left\{(u,v) \mid \exp_{f(y)}^{-1}f(x)-u \not<_{C_{f(y)}\setminus\{0\}} 0, \, v_i \leq g_i(x), \, i=1,\ldots,l, \;\;\forall x \in M \right\},
	\end{equation}
	where $\text{cl}$ denotes closure and $v_i$ is the ith-component of $v$. We call the set $\varepsilon(y)$ defined in \eqref{EQ2} the extended image of problem \eqref{question}. Indeed, the efficient solution can also be characterised by the extended image $\varepsilon(y)$.
	
	\begin{Theorem}\label{th4.1}
		Let $y$ be a point in $K$. The following statements are equivalent:
		\begin{itemize}
			\item[(1)]$y$ is the efficient solution of problem \eqref{question};
			\item[(2)]$\mathcal{H}\cap\varepsilon(y)=\emptyset $;
			\item[(3)]$\mathcal{H}\cap\mathcal{K}=\emptyset $.
		\end{itemize}
	\end{Theorem}
	
	\begin{proof}
		$(1){\Rightarrow}(2)$ For any$(\bar{u},\bar{v})\in\varepsilon(y)$, and $y$ is the efficient solution of problem \eqref{question}, we know that \eqref{EQ3} holds. It implies that
		\begin{equation}\label{EQ7}
			0\not<_{{C_{f(y)}\setminus\{0\}}}\exp_{f(y)}^{-1}f(x)\Rightarrow-\bar{u}\not<_{{C_{f(y)}\setminus\{0\}}}\exp_{f(y)}^{-1}f(x)-\bar{u}
		\end{equation}
		and
		\begin{equation}\label{EQ8}
			\exp_{f(y)}^{-1}f(x)-\bar{u}\not<_{C_{f(y)}\setminus\{0\}}0.
		\end{equation}
		
		It follows from \eqref{EQ7} and \eqref{EQ8} that
		$$0\not<_{C_{f(y)}\setminus\{0\}}\bar{u}.$$
		It means $\bar{u}\notin C_{f(y)}\setminus\{0\}$ and $(\bar{u},\bar{v})\notin\mathcal{H}$. 
		
		$(2){\Rightarrow}(3)$ Since $0\notin C_{f(y)}\setminus\{0\}$, for any $x\in M$, we have
		$$\exp_{f(y)}^{-1}f(x)-\exp_{f(y)}^{-1}f(x)=0\not<_{C_{f(y)}\setminus\{0\}}0.$$
		If $u=\exp_{f(y)}^{-1}f(x)$, for any $x\in M$, we have 
		$$\exp_{f(y)}^{-1}f(x)-u\not<_{C_{f(y)}\setminus\{0\}}0.$$
		Also, it is easy to see that for any $x\in M$, $i=1,...,l$,
		$$v_i=g_i(x)\leq g_i(x). $$
		Now combining $\mathcal{K}\subset\varepsilon(y)$ and (2), we know that $\mathcal{H}\cap\mathcal{K}=\emptyset $.
		
		$(3){\Rightarrow}(1)$ It is clear that $\mathcal{H}\cap\mathcal{K}=\emptyset $ implies  at least one of the following cases:
		\begin{itemize}
			\item[(i)] $\exp_{f(y)}^{-1}f(x)\notin C_{f(y)}\setminus\{0\}$ holds for any $x\in K$, or
			\item[(ii)]$g_i(x)<0$ for all $i=1,...,l$ and $x\in M$.
		\end{itemize}
		However, the feasible set $K\neq\emptyset $ and (ii) is false. This means that $y$ is the efficient solution of problem \eqref{question} by (i).
	\end{proof}
	
	How to introduce a suitable separation function is the key to the ISA method, which we will discuss later.

	\section{Separation Function and Applications}
\setcounter{equation}{0} 

 \subsection{Separation Functions}\noindent

So far, it is clear that $y$ is an efficient solution to problem \eqref{question} iff  $\mathcal{H}\cap\mathcal{K}=\emptyset $. However, it is not easy to prove it directly. One way to prove it indirectly is to obtain the existence of a separating function such that the two disjoint level sets of this separating function contain $\mathcal{H}$ and $\mathcal{K}$. Weakly separation functions, in particular regular weak separation functions, are widely used to study gap functions, optimality conditions and pairwise outcomes. We draw on previous studies to give two specific classes of separation functions based on the GOP in this paper.
	
	\begin{Definition}\label{de5.1}
		The function $\omega_i:T_{f(y)}M\times R^l\times\Pi_i\to R$ is said to be weak separation function with respect to $\mathcal{H}$ for $i=1,2,$ iff it satisfies
		$$
		\begin{aligned}
			& \mathrm{lev}_{\geq 0}\omega_i(\cdot;\pi) \supseteq \mathcal{H}, \;\;\forall \pi \in \Pi_i; \\
			& \bigcap_{\pi \in \Pi_i} \mathrm{lev}_{> 0}\omega_i(\cdot;\pi) \subseteq \mathcal{H}, &
		\end{aligned}
		$$
		where $\Pi_{i}$ is a set of parameters to be specified case by case. 
	\end{Definition}
	
	\begin{Definition}\label{de5.2}
		The function $\omega_i:T_{f(y)}M\times R^l\times\Pi_i\to R$ is said to be regular weak separation functions with respect to $\mathcal{H}$ for $i=1,2,$ iff it satisfies
		$$\bigcap_{\pi\in\Pi_i}\mathrm{lev}_{>0}\omega_i(\cdot;\pi)=\mathcal{H}.$$
	\end{Definition}
	
	Moreover, we consider a special class of weak separation functions $\omega_1:T_{f(y)}M\times R^l\times\Pi_i\to R$. If $\Pi_1=( \theta,\lambda)$, then
	\begin{equation}\label{EQ9}
		\omega_1(u,v;\theta,\lambda)=\langle\theta,u\rangle+\langle\lambda,v\rangle,
	\end{equation}
	\begin{equation}\label{EQ10}
		\theta\in C_{f(y)}^*=\{\theta\in T_{f(y)}M|\langle\theta,b\rangle\geq0,\;\;\forall b\in C_{f(y)}\subset T_{f(y)}M\},
	\end{equation}
	$$\lambda\in V^*=\{l\in R^l|\langle l,r\rangle\geq0,\;\;\forall r\in R_+^l\}, $$
	where $\theta $ and $\lambda$ are two parameters. We have to find suitable parameters to construct the separation function.
	
	\begin{Theorem}\label{th5.1}
		If $W(\theta,\lambda)$ is the level set of $\omega_1(u,v;\theta,\lambda),$ i.e.
		 $$W(\theta,\lambda)=\mathrm{lev}_{>0}\omega_1(u,v;\theta,\lambda)=\{(u,v)|\omega_1(u,v;\theta,\lambda)>0\},$$ then $\mathcal{H}\subset W(\theta,\lambda)$ and
		\begin{equation}\label{EQ12}
			\mathcal{H}=\bigcap_{{\theta\in C_{f(y)}^{*}\setminus\{0\}}}W(\theta,\lambda).
		\end{equation}
	\end{Theorem}
	
	\begin{proof}
		First, we prove $\mathcal{H}\subset W(\theta,\lambda)$. It is obviously that $(u,v)\in\mathcal{H}$ is equivalent to $u\in C_{f(y)}\setminus\{0\},$ $v\in R_+^l$. For any $\theta\in C_{f(y)}^*\setminus\{0\}, \lambda\in V^{*}$ and $(u,v)\in\mathcal{H},$ we have $\langle\theta,u\rangle>0$, $\langle\lambda,v\rangle\geq0,$  and 		$$\omega_1(u,v;\theta,\lambda)=\langle\theta,u\rangle+\langle\lambda,v\rangle>0.$$
		It follows from the definition of $W(\theta,\lambda)$ that $\mathcal{H}\subset W(\theta,\lambda)$ is true.
		
		Next we show that \eqref{EQ12} holds. It is only necessary to prove for any $(\hat{u},\hat{v})\notin\mathcal{H}$, there exists $\hat{\theta}\in C_{f(y)}^*\setminus\{0\}$ and $\hat{\lambda}\in V^*$ such that
		\begin{equation}\label{EQ13}
			(\hat{u},\hat{v})\notin W(\hat{\theta},\hat{\lambda}).
		\end{equation}
		
		Thus, we have at least one of the following two cases holds:
		\begin{itemize}
			\item[(i)] $\hat{u}\notin C_{f(y)}\setminus\{0\}$ or 
			\item[(ii)] $\hat{v}\notin R_+^l$.
		\end{itemize}
		
		If (i) holds, let $\hat{\lambda}=0\in V^{*}$. If $\hat{u}=0$, for any $\hat{\theta}\in C_{f(y)}^*\setminus\{0\}$, we have $$\omega_1(\hat{u},\hat{v};\hat{\lambda},\hat{\theta})=\langle\hat{\theta},0\rangle+\langle0,\hat{v}\rangle=0\not>0.$$
		Then  $(\hat{u},\hat{v})\notin W(\hat{\theta},\hat{\lambda})$.
		
		If  $\hat{u}\neq0$, then  $\hat{u}\notin C_{f(y)}$ and hence exists  $\hat{\theta}\in C_{f(y)}^*\setminus\{0\}$ such that
		\begin{equation}\label{EQ14}
			\langle\hat{\theta},u\rangle\leq0,\;\;\forall u\notin C_{f(y)}.
		\end{equation}
		
		If there exists $\hat{\theta}\in C_{f(y)}^*\setminus\{0\}$ such that
		$$\omega_1\big(\hat{u},\hat{v};\hat{\lambda},\hat{\theta}\big)=\langle\hat{\theta},\hat{u}\rangle+\langle0,\hat{v}\rangle=\langle\hat{\theta},\hat{u}\rangle\leq0,$$
		$$(\hat{u},\hat{v})\notin W(\hat{\theta},\hat{\lambda}),$$
		and it leads to a contradiction. Otherwise, if \eqref{EQ14} has no solution, then one has
		\begin{equation}\label{EQ15}
			\langle\theta,u\rangle>0,\;\;\forall u\notin C_{f(y)},\;\;\forall\theta\in C_{f(y)}^{*}\setminus\{0\}.
		\end{equation}
	
		Since $\theta\in C_{f(y)}^*\setminus\{0\}\subset C_{f(y)}^*$ and \eqref{EQ10}, we have
		\begin{equation}\label{EQ16}
			\langle\theta,u\rangle\geq0,\;\;\forall u\in C_{f(y)},\;\;\forall\theta\in C_{f(y)}^*\setminus\{0\}.
		\end{equation}
		Combining \eqref{EQ15} and \eqref{EQ16}, we have
		\begin{equation}\label{EQ17}
			\langle\theta,u\rangle\geq0,\;\;\forall u\in T_{f(y)}M,\;\;\forall\theta\in C_{f(y)}^*\setminus\{0\}.
		\end{equation}
		
		However, setting $u=-\theta\neq0$, then we have  $\langle\theta,u\rangle=\langle\theta,-\theta\rangle=-\|\theta\|^2<0$, which is a contradicts to \eqref{EQ17}. 
		
		If (ii) holds, there exists at least one component $\hat{v}_{i_0}$ of $\hat{v}$ such that $\hat{v}_{i_0}<0$.
		Take $\check{\lambda}_{i_0}>0$ such that $\check{\lambda}=(0,\ldots,0,\check{\lambda}_{i_0},0,\ldots,0)$.

There is no doubt that $\check{\lambda}$ belongs to the set $V^{*},\;\;\langle\check{\lambda},\hat{v}\rangle=\check{\lambda}_{i_0}\hat{v}_{i_0}<0$. Let $\check{\theta}\in C_{f(y)}^{*}\setminus\{0\}$ be a given parameter and for any $\alpha>0$ be any given positive number. Then $\alpha\check{\theta}$ also belongs to $C_{f(y)}^*\setminus\{0\}$. Setting $$0<\alpha<\frac{-\check{\lambda}_{i_0}\hat{v}_{i_0}}{\langle\check{\theta},\hat{u}\rangle},$$ 
		one has $$\omega_{1}(\hat{u},\hat{v};\alpha\check{\theta},\check{\lambda})=\alpha\langle\check{\theta},\hat{u}\rangle+\check{\lambda}_{{i_{0}}}\hat{v}_{{i_{0}}}<0.$$ 
		This shows that
		$(\hat{u},\hat{v})\notin W(\alpha\check{\theta},\check{\lambda})$, which is a contradiction.
	\end{proof}
	
	Next, we introduce another special class of nonlinear weak separation functions. This weak separating function can be expressed as the sum of two partial functions. The set of parameters is given by $$\Pi_2=\mathrm{T}\times\Gamma=C_{f(y)}^{*}\times(-\text{int}R_{+}^{l}\cup\{0_{l}\})$$
	 with $\pi=(\varphi,\gamma)\in\Pi_2.$ For any $\varphi\in \mathrm{T}$ and $\gamma\in\Gamma,$ the mapping $\omega_{2}:T_{f(y)}M\times R^l\times\Pi_2\to R$ is defined by:
	\begin{equation}\label{EQ18}
		\omega_2(u,v;\pi)=\omega_2(u,v;\varphi,\gamma)=\widetilde{\omega}(u;\varphi)+\underline{\omega}(v;\gamma).
	\end{equation}
	
	In particular, $\widetilde{\omega}(u;\varphi)\colon T_{f(y)}M\times\mathrm{T}\to\mathrm{R},\underline{\omega}(\nu;\gamma) : R^l\times\Gamma\to R$ is defined as:
	\begin{equation}\label{EQ19}
		\widetilde{\omega}(u;\varphi)=-\Delta_{R_+}(\langle\varphi,u\rangle),
	\end{equation}
	\begin{equation}\label{EQ20}
		\underline{\omega}(v;\gamma)=\begin{cases}-\xi_{R_+^l,\gamma}(v) ,& \text{if } \gamma \in -\text{int}R_+^l, \\0, & \text{if } \gamma = 0_l.\end{cases}
	\end{equation}
   For any $v\in {R}^l, \gamma\in\Gamma$ and $\alpha\geq0,$ there exists $\gamma_\alpha\in\Gamma$ such that 
	\begin{equation}\label{EQ21}
		\alpha\underline{\omega}(v;\gamma)=\underline{\omega}(v;\gamma_\alpha),
	\end{equation}
	and it implies that there exists $\tilde{\gamma}\in\Gamma$ such that
	\begin{equation}\label{EQ22}
		\underline{\omega}(\cdot;\tilde{\gamma})\equiv0.
	\end{equation}
	
	\begin{Theorem}\label{th5.2}
		The family of functions $\omega_{2}$ defined by \eqref{EQ19}, \eqref{EQ20} ,\eqref{EQ21} and \eqref{EQ22} is a class of 
		weak separation functions with respect to $\mathcal{H}$, where the parameter set $\Pi_{2}=\mathrm{T}\times\Gamma=C_{f(y)}^{*}\times-\text{int}R_{+}^{l}\cup\{0_{l}\}$.
	\end{Theorem}
	
	\begin{proof}
		For any $\varphi\in{T},u\in C_{f(y)}\setminus\{0\}$, there is $\langle\varphi,u\rangle\geq0$. By Lemma \ref{pr2.4} (iii) and (iv) it follows that
		\begin{equation}\label{EQ23}
			\widetilde{\omega}(u;\varphi)\geq0.
		\end{equation}
		By Lemma \ref{pr2.5} (ii), for any $\gamma\in-\text{int}R_+^l,v\in{R}_+^l$, we have $-\xi_{{R}_+^l,\gamma}(\nu)\geq0$. Set $\gamma=0_{l},$ and therefore  
		\begin{equation}\label{EQ24}
			\underline{\omega}(v;\gamma)\geq0.
		\end{equation}
		Combining with \eqref{EQ23} and \eqref{EQ24}, we have 
		\begin{equation}
			\omega_2(u,v;\varphi,\gamma)=\widetilde{\omega}(u;\varphi)+\underline{\omega}(v;\gamma)\geq0.
		\end{equation}
		Then for any $(\varphi,\gamma)\in\Pi_2,$ one has
		\begin{equation}
		\mathrm{lev}_{\geq0}\omega_2(u,v;\varphi,\gamma)\supseteq\mathcal{H}.
		\end{equation}
		It is only need to prove
		\begin{equation}\label{EQ25}
			\bigcap_{\pi\in\Pi_2}\mathrm{lev}_{>0}\omega_2(u,v;\varphi,\gamma)\subseteq\mathcal{H}.
		\end{equation}
It is necessary to demonstrate that there exists $(\varphi,\gamma)\in\Pi_2$ such that $\omega_2(u,v;\varphi,\gamma)\leq0$ when $(u,v)\notin\mathcal{H}$.
		
		For $(u,v)\notin\mathcal{H}$, we discuss it in the following two cases:
			\begin{itemize}
			\item[(i)] $u\notin C_{f(y)}\setminus\{0\}$, and $v\in{R}_{+}^{l}$ or 
			\item[(ii)]  $u\in C_{f(y)}\setminus\{0\}$, and $v\notin{R}_+^l$.
		\end{itemize}
		
			If (i) holds, there exists $\varphi\in C_{f(y)}^*$, such that $\langle\varphi,u\rangle\leq0$, and $t\varphi\in C_{f(y)}^*$ when $t>0$. By Lemma \ref{pr2.4}, it follows that $\tilde{\omega}(u;\varphi) = -\Delta_{R_+}(\langle t\varphi, u\rangle)$ converges to negative infinity as $t$ approaches positive infinity. Let $\gamma=0_{l}$, we have  $\underline{\omega}(v;\gamma)=0$, so there exists $(\varphi,\gamma)\in\Pi_2$ such that $\omega_2(u,v;\varphi,\gamma)\leq0$.
		
			If (ii) holds, by Lemma \ref{pr2.5}(ii), if $\xi_{R_{+}^{l},\gamma}(v)<0$, then $v\in{R}_+^l$, which contradicts the assumption $\xi_{R_{+}^{l},\gamma}(v)\geq0,\underline{\omega}(v;\gamma)\leq0$. Let $\varphi=0$, we have
		$\tilde{\omega}(u;\varphi)=0$,    $\omega_2(u,v;\varphi,\gamma)\leq0$.
	\end{proof}
	
			
	\subsection{Characterizations for optimality conditions}\noindent
	
	In section part, we discuss a class of regular weak separation functions \eqref{EQ9}, and give a sufficient optimality condition for problem \eqref{question}.
	
	\begin{Theorem}\label{th6.1}
		Let $y\in K$, for any $x\in M,$ there exists $\theta_0\in C_{f(y)}^*\setminus\{0\}$ and $\lambda_0\in V^*$ such that 
		\begin{equation}\label{6.1}
			\langle\theta_0,\exp_{f(y)}^{-1}f(x)\rangle+\langle\lambda_0,g(x)\rangle\leq0,
		\end{equation}
		then $y$ is an efficient solution to the GOP \eqref{question}.
	\end{Theorem}
	
	\begin{proof}
		From Theorem \ref{th4.1}, we know that $y\in K$  is an efficient solution of problem \eqref{question} iff
		\begin{equation}\label{6.2}
			\mathcal{H}\cap\mathcal{K}=\emptyset.
		\end{equation}
		By Theorem \ref{th5.1}, we have  
		\begin{equation}\label{6.3}
			\mathcal{H}=\bigcap_{\theta\in C_{f(y)}^{*}\setminus\{0\}}\mathrm{lev}_{>0}\omega_{1}(u,v;\theta,\lambda).
		\end{equation}
		Therefore, we only need to prove that there exist $\theta_0\in C_{f(y)}^*\setminus\{0\}$ and $\lambda_0\in V^*$ such that
		\begin{equation}\label{6.4}
			\mathcal{K}\subset\mathrm{lev}_{\leq0}\omega_1(u,v;\theta_0,\lambda_0).
		\end{equation}
		It is easy to see that \eqref{6.4} is equivalent to \eqref{6.1}. This completes the proof.
	\end{proof}
	
	Next, we introduce another class of weak separation functions, which is denoted by $\omega_2(u,v;\varphi,\gamma)$, and give a sufficient optimality condition for \eqref{question}.
	
	\begin{Theorem}\label{th6.2}
		For any $x\in M$, $y\in K$ is an efficient solution to problem \eqref{question} if there exists $\varphi\in C_{f(y)}^*$ and $\gamma\in-\text{int}{R}_+^l\cup\{0_l\}$, such that
		
		\begin{equation*}
			0\geq\omega_2(M_y(x);\varphi,\gamma)= \left\{
			\begin{aligned}
				-\Delta_{{R}_+}\big(\big\langle\varphi,\exp_{f(y)}^{-1}f(x)\big\rangle\big)-\xi_{{R}_+^l,\gamma}\big(g(x)\big),\, & \gamma\in-\text{int}{R}_+^l,\\
				-\Delta_{{R}_+}\big(\big\langle\varphi,\exp_{f(y)}^{-1}f(x)\big\rangle\big),\, & \gamma=0_l.\\
			\end{aligned}
			\right.
		\end{equation*}		
		In other words, if there exists $(\varphi,\gamma)\in{T}\times\Gamma $ such that
		\begin{equation}\label{6.5}
			\omega_2(M_y(x);\varphi,\gamma)=\widetilde{\omega}(\exp_{f(y)}^{-1}f(x);\varphi)+\underline{\omega}(g(x);\gamma)\leq0,
		\end{equation}
		then $y$ is an efficient solution to problem \eqref{question}.
	\end{Theorem}
	
	\begin{proof} From Theorem \ref{th4.1}, we know that $y\in K$ is an efficient solution of problem \eqref{question} iff \eqref{6.2} holds. From Theorem \ref{th5.2}, we have $$\bigcap_{\pi\in\Pi_2}\text{lev}_{>0}\omega_2(u,v;\varphi,\gamma)\subseteq\mathcal{H}.$$ 
		It is easy to get $\mathcal{K}\subset\mathrm{lev}_{\leq0}\omega_2(u,v;\varphi,\gamma)$	by \eqref{6.5}. This completes the proof.
	\end{proof}
	
	Next, we apply the two classes of separation functions to establish sufficient and necessary optimality conditions for GOP, in particular saddle-point type conditions.
	
	We introduce two classes of generalized Lagrangian functions $\mathcal{L}_y^1:M\times\mathcal{C}_{f(y)}^*\setminus\{0\}\times V^*\to R$ and  $\mathcal{L}_y^2:M\times\mathcal{T}\times\Gamma\to R$, which are constructed by weak separation functions with respect to $y\in K$:
	\begin{equation}
	\mathcal{L}_y^1(x;\theta,\lambda)=-\langle\theta,\exp_{f(y)}^{-1}f(x)\rangle-\langle\lambda,g(x)\rangle,
	\end{equation}
	\begin{equation}\label{l2}
		\mathcal{L}_y^2(x;\varphi,\gamma)=-\tilde{\omega}\big(\exp_{f(y)}^{-1}f(x);\varphi\big)-\underline{\omega}(g(x);\gamma).
	\end{equation}
	
	\begin{Definition}\label{de6.1}
	 The point $(y,\bar{\lambda})\in M\times V^*$ is a generalized saddle point of $\mathcal{L}_y^1(x;\theta,\lambda)$ iff for any $x\in M$ and $\lambda\in V^{*}$ there exists $\bar{\theta}\in C_{f(y)}^*\setminus\{0\},$ such that
		\begin{equation}
			\mathcal{L}_{y}^{1}(y;\bar{\theta},\lambda)\leq\mathcal{L}_{y}^{1}(y;\bar{\theta},\bar{\lambda})\leq\mathcal{L}_{y}^{1}(x;\bar{\theta},\bar{\lambda}).
		\end{equation}
	\end{Definition}

	\begin{Definition}\label{de6.2}
		The point $(y,\bar{\gamma})\in M\times\Gamma $ is a generalized saddle point of $\mathcal{L}_y^2(x;\varphi,\gamma)$ iff for any $x\in M$ and $\gamma\in\Gamma,$  there exists $\bar{\varphi}\in C_{f(y)}^*$ such that
		\begin{equation}\label{6.10}
			\mathcal{L}_y^2(y;\bar{\varphi},\gamma)\leq\mathcal{L}_y^2(y;\bar{\varphi},\bar{\gamma})\leq\mathcal{L}_y^2(x;\bar{\varphi},\bar{\gamma}). 
		\end{equation}
	\end{Definition}
	
	Now we can characterize the separation of two sets in an image space by using the extence of saddle point of generalized Lagrange functions $\mathcal{L}_{y}^{i}, i=1,2.$
	
	\begin{Theorem}\label{th6.3}
		Let $y\in K$ be a given point, $(y,\bar{\lambda})$ is the generalized saddle point of $\mathcal{L}_y^1(x;\bar{\theta},\lambda)$ iff there exists $\bar{\theta}\in C_{f(y)}^*\setminus\{0\}$ and $\bar{\lambda}\in V^*$, such that \eqref{6.1} holds.
	\end{Theorem}
	
	\begin{proof}
		$\text{``}\Rightarrow\text{''}$
	Since $(y,\bar{\lambda})$ is a saddle point,
	$$\begin{aligned}
		& -\langle\bar{\theta}, \exp_{f(y)}^{-1} f(y)\rangle - \langle\lambda, g(y)\rangle \\
		& \leq -\langle\bar{\theta}, \exp_{f(y)}^{-1} f(y)\rangle - \langle\bar{\lambda}, g(y)\rangle \\
		& \leq -\langle\bar{\theta}, \exp_{f(y)}^{-1} f(x)\rangle - \langle\bar{\lambda}, g(x)\rangle.
	\end{aligned}
	$$
	Taking $\lambda=0$ in $\mathcal{L}_y^1(\mathrm{y};\bar{\theta},\lambda)\leq\mathcal{L}_y^1(\mathrm{y};\bar{\theta},\bar{\lambda})$, and we have $$\langle\bar{\lambda},g(y)\rangle\leq0.$$
	On the other hand, since $y\in K$, $\bar{\lambda}\in V^*$, we have $\langle\bar{\lambda},g(y)\rangle\geq0$, and $\langle\bar{\lambda},g(y)\rangle=0$.
	
	Therefore, for any $x\in M$ one has $\mathcal{L}_y^1(y;\bar{\theta},\bar{\lambda})\leq\mathcal{L}_y^1(x;\bar{\theta},\bar{\lambda})$, for any $x\in M$.
	This implies that
	$$\begin{aligned}
		& -\langle\bar{\theta}, \exp_{f(y)}^{-1} f(y)\rangle - \langle\bar{\lambda}, g(y)\rangle \leq -\langle\bar{\theta}, \exp_{f(y)}^{-1} f(x)\rangle - \langle\bar{\lambda}, g(x)\rangle \\
		& \Leftrightarrow \langle\bar{\theta}, \exp_{f(y)}^{-1} f(x)\rangle + \langle\bar{\lambda}, g(x)\rangle \leq \langle\bar{\theta}, \exp_{f(y)}^{-1} f(y)\rangle + \langle\bar{\lambda}, g(y)\rangle = 0.
	\end{aligned}
	$$	
		
	$\text{``}\Leftarrow\text{''}$	By setting $x=y$ in \eqref{6.1}, we have  
	$$\langle\bar{\theta},\exp_{f(y)}^{-1}f(y)\rangle+\langle\bar{\lambda},g(y)\rangle=\langle\bar{\lambda},g(y)\rangle\leq0.$$
	On the other hand, since $y\in K,\bar{\lambda}\in V^*$, we have $\langle\bar{\lambda},g(y)\rangle\geq0$, and so
	\begin{equation}\label{6.11}
		\left\langle\bar{\lambda},g(y)\right\rangle=0.
	\end{equation}
	
	For any $\lambda\in V^{*}$, we have
	$$\langle\lambda,g(y)\rangle\geq0=\langle\bar{\lambda},g(y)\rangle,$$ 
	and for any $x\in M$, we have
	$$\langle\bar{\theta},\exp_{f(y)}^{-1}f(y)\rangle+\langle\bar{\lambda},g(y)\rangle\leq\langle\bar{\theta},\exp_{f(y)}^{-1}f(y)\rangle+\langle\lambda,g(y)\rangle.$$
	Therefore $\mathcal{L}_y^1(\mathrm{y};\bar{\theta},\lambda)\leq\mathcal{L}_y^1(\mathrm{y};\bar{\theta},\bar{\lambda})$.
	
	Combining \eqref{6.1} and \eqref{6.11}, for any $x\in M$,
	$$\langle\bar{\theta},\exp_{f(y)}^{-1}f(x)\rangle+\langle\bar{\lambda},g(x)\rangle-\langle\bar{\lambda},g(y)\rangle\leq0,$$
	i.e.,  for any $x\in M$,
	$$\langle\bar{\theta},\exp_{f(y)}^{-1}f(x)\rangle+\langle\bar{\lambda},g(x)\rangle\leq0+\langle\bar{\lambda},g(y)\rangle=\langle\bar{\theta},\exp_{f(y)}^{-1}f(y)\rangle+\langle\bar{\lambda},g(y)\rangle.$$
	Then we have  $\mathcal{L}_y^1(y;\bar{\theta},\bar{\lambda})\leq\mathcal{L}_y^1(x;\bar{\theta},\bar{\lambda}).$

	\end{proof}

We study the dual problem in the rest of this section. We give another equivalent characterization of \eqref{6.1} by the duality. 

For each fixed pair $(\theta,\lambda)$, the corresponding dual problem is defined as follows:
	\begin{equation}\label{6.12}
		\begin{aligned}\max&\quad\omega_1(u,v;\theta,\lambda),\\\mathrm{s.t~}&\;\;\;\;(u,v)\in\mathcal{K},\end{aligned}
	\end{equation}
	where $\omega_{1}$ is given by \eqref{EQ9} , and $(\tilde{u},\tilde{v})\in\mathcal{K}$ is the maximum point of \eqref{6.12} iff
	\begin{equation}\label{6.13}
		\omega_1(\tilde{u},\tilde{v};\theta,\lambda)\geq\omega_1(u,v;\theta,\lambda),\;\;\forall(u,v)\in\mathcal{K}.
	\end{equation}
	
	\begin{Lemma}\label{le6.1}
		If there exists a maximum point in \eqref{6.12}, then for any $\theta\in C_{f(y)}^*\setminus\{0\}$ and $\lambda\in V^*,$ we have
		\begin{equation}\label{6.14}
			\max_{(u,v)\in\mathcal{K}}\omega_1(u,v;\theta,\lambda)\geq0.
		\end{equation}
	\end{Lemma}
	
	\begin{proof}
		Let $\tilde{u}=\exp_{f(y)}^{-1}f(y)=0,\tilde{v}=g(y)$, it is easy to see that $(\tilde{u},\tilde{v})\in\mathcal{K}$ and 
		$$\omega_1(\tilde{u},\tilde{v};\theta,\lambda)=\langle\theta,\exp_{f(y)}^{-1}f(y)\rangle+\langle\lambda,g(y)\rangle=\langle\lambda,g(y)\rangle .$$
		Since $\lambda\in V^{*}$, $\omega_1(\tilde{u},\tilde{v};\theta,\lambda)\geq0$. Therefore, if $(\tilde{u},\tilde{v})$ is the maximum point, then \eqref{6.14} holds. If not, let $(u^0,v^0)\in\mathcal{K}$ be the maximum point of \eqref{6.14}, we have 
		$$\omega_1(u^0,v^0;\theta,\lambda)>\omega_1(\tilde{u},\tilde{v};\theta,\lambda)\geq0.$$
		This completes the proof. 
	\end{proof}
	
	From the above Lemma, we have 
	$$\omega_1(u^0,v^0;\theta,\lambda)-\omega_1(\tilde{u},\tilde{v};\theta,\lambda)>0,$$
	and $-\omega_1(u^0,v^0;\theta,\lambda)<0$.
	
	After discussing this class of regular weak separation functions, we can further explore the duality theory. Under the assumption of regular weak separation, we are able to demonstrate that the duality gap equals zero. First, we introduce the definition of image duality gap.
	
	\begin{Definition}\label{de6.3}
		${\Omega}$ is said to be a image duality gap iff it is the optimal value for the following problem:
		\begin{equation}\label{minmax}
			\min_{\theta\in C_{f(y)}^*\setminus\{0\}}\max_{(u,v)\in\mathcal{K}}\omega_1(u,v;\theta,\lambda).
		\end{equation}
		
	\end{Definition}
	This allows us to present the following duality theorem.
	\begin{Theorem}\label{th6.5}
		The image duality gap $\Omega=0$ iff for any $(u,v)\in\mathcal{K}$, there exists $\bar{\theta}\in C_{f(y)}^*\setminus\{0\}$ and $\bar{\lambda}\in V^*$, such that
		\begin{equation}\label{6.16}
			\omega_1(u,v;\tilde{\theta},\tilde{\lambda})\leq0.
		\end{equation}
	\end{Theorem}
	
	\begin{proof}
			$\text{``}\Rightarrow\text{''}$ Suppose that $\Omega = 0$, then there exists $(\tilde{\theta},\tilde{\lambda})\in C_{f(y)}^*\setminus\{0\}\times V^*$ such that $$\max_{(u,v)\in\mathcal{K}}\omega_1(u,v;\tilde{\theta},\tilde{\lambda})=0.$$
		By the definition of maximum, there exists $(\tilde{u},\tilde{v})\in\mathcal{K}$, for any $(u, v)\in\mathcal{K}$ such that $$\omega_1(u,v;\tilde{\theta},\tilde{\lambda})\leq\omega_1(\tilde{u},\tilde{v};\tilde{\theta},\tilde{\lambda})=0.$$

		$\text{``}\Leftarrow\text{''}$ Set $x=y$, and $$(\tilde{u},\tilde{v})=\left(\exp_{f(y)}^{-1}f(y),g(y)\right)=(0,g(y))\in\mathcal{K}.$$
		Since $y\in K$ and $\tilde{\lambda}\in V^*$, we can get $$\omega_1(\tilde{u},\tilde{v};\tilde{\theta},\tilde{\lambda})=\langle\tilde{\lambda},g(y)\rangle\geq0.$$ 
		Because $\omega_1(\tilde{u},\tilde{v};\tilde{\theta},\tilde{\lambda})=0$, this means that $(\tilde{u},\tilde{v})$ is the maximum point of $\omega_1(u, v; \tilde{\theta},\tilde{\lambda})$, and $z=0$ is the optimal value for problem \eqref{6.12}. Besides, by Lemma \ref{le6.1}, we know that $0$ is also the optimal value for \eqref{minmax}, that is $\Omega = 0$.
	\end{proof}
	
	Next, we study anoher generalized Lagrangian function defined in \eqref{l2}, and present the relationship between the separation in the IS and the existence of saddle points of $\mathcal{L}_y^2(x;\varphi,\gamma)$.
	
	\begin{Theorem}\label{th6.6}
		For any given $y\in K$, $(y,\bar{\gamma})$ is the generalized saddle point of $\mathcal{L}_y^2(x;\bar{\varphi},\gamma)$ iff there exists $\bar{\varphi}\in C_{f(y)}^*$ and $\bar{\gamma}\in\Gamma^*,$ such that \eqref{6.5} holds.
	\end{Theorem}

	\begin{proof}
		$\text{``}\Rightarrow\text{''}$: Since $(y,\bar{\gamma})$ is the generalized saddle point of $\mathcal{L}_y^2(x;\bar{\varphi},\gamma),$ we have 
		\begin{equation}\label{6.17}
			\begin{split}
				& -\tilde{\omega}(\operatorname{\exp}_{f(y)}^{-1}f(y);\bar{\varphi}) - \underline{\omega}(g(y);\gamma)  \\
				& \leq -\tilde{\omega}(\operatorname{\exp}_{f(y)}^{-1}f(y);\bar{\varphi}) - \underline{\omega}(g(y);\bar{\gamma}) \\
				& \leq -\tilde{\omega}(\operatorname{\exp}_{f(y)}^{-1}f(x);\bar{\varphi}) - \underline{\omega}(g(x);\bar{\gamma}).
			\end{split}
		\end{equation}
		By \eqref{EQ22}, there exists $\tilde{\gamma}\in\Gamma $ such that $\underline{\omega}(\cdot;\tilde{\gamma})\equiv0$. Taking $\gamma=\tilde{\gamma}$ in the first inequality of \eqref{6.10}, we have $\underline{\omega}(g(y);\bar{\gamma})\leq0.$ 
		
		Combined with the fact that $y\in K$ implies that $g(y)\in D$ and $\underline{\omega}(g(y);\bar{\gamma})\geq0$,we can deduce that $\underline{\omega}(g(y);\bar{\gamma})=0$. Therefore, from the second inequality in \eqref{6.17} and $\underline{\omega}(g(y);\bar{\gamma})=0$. we get that
		$$-\tilde{\omega}\big(\exp_{f(y)}^{-1}f(y);\bar{\varphi}\big)\leq-\tilde{\omega}\big(\exp_{f(y)}^{-1}f(x);\bar{\varphi}\big)-\underline{\omega}(g(x);\bar{\gamma}\big).$$
		Since $y\in K$, it follows that $\Delta_{R_+}\big(\big\langle\bar{\varphi},\exp_{f(y)}^{-1}f(y)\big\rangle\big)=0$, i.e., $\widetilde{\omega}\big(\exp_{f(y)}^{-1}f(y);\bar{\varphi}\big)=0$, i.e.,
		$$0\leq-\tilde{\omega}(\exp_{f(y)}^{-1}f(x);\bar{\varphi})-\underline{\omega}(g(x);\bar{\gamma}).$$
		Therefore \eqref{6.5} holds. 
		
		$\text{``}\Leftarrow\text{''}$: Setting $x = y$ in \eqref{6.5}, we get $\underline{\omega}(g(y);\bar{\gamma})\leq0$, and since $y\in K$, then $g(y) \in R_+^l$. It implies that $\underline{\omega}(g(y);\bar{\gamma})\geq0$, so $\underline{\omega}(g(y);\bar{\gamma})=0$. For any $x\in M$, by \eqref{6.5} and $\underline{\omega}(g(y);\bar{\gamma})=0$, we have
		$$-\tilde{\omega}\big(\exp_{f(y)}^{-1}f(y);\bar{\varphi}\big)-\underline{\omega}(g(y);\bar{\gamma})\leq-\tilde{\omega}\big(\exp_{f(y)}^{-1}f(x);\bar{\varphi}\big)-\underline{\omega}(g(x);\bar{\gamma}),$$
		that is $\mathcal{L}_y^2(y;\bar{\varphi},\bar{\gamma})\leq\mathcal{L}_y^2(x;\bar{\varphi},\bar{\gamma}).$
		
		On the other hand, since $y\in K$, and $g(y)\in D$, then for any $\gamma\in\Gamma$ we have $\underline{\omega}(g(y);\gamma)\geq0$. Therefore, for any $x\in M$,
		$$-\tilde{\omega}\Big(\exp_{f(y)}^{-1}f(y);\bar{\varphi}\Big)-\underline{\omega}(g(y);\;\gamma)\leq-\tilde{\omega}\Big(\exp_{f(y)}^{-1}f(y);\bar{\varphi}\Big)-\underline{\omega}(g(y);\;\bar{\gamma}),$$
		that is $\mathcal{L}_y^2(y;\bar{\varphi},\gamma)\leq\mathcal{L}_y^2(y;\bar{\varphi},\bar{\gamma}),$ so $(y,\bar{\gamma})$ is a generalized saddle point of the function $\mathcal{L}_y^2(x;\bar{\varphi},\gamma)$ on $M\times\Gamma $. 
	\end{proof}

In this section, we explored the generalized Lagrange functions generated by the two class of separation functions. We deduced the general necessary and sufficient conditions for optimality. Additionally, for the regular weak separation functions, we delved into the duality theory and demonstrated that the presence of regular separation ensures the existence of a zero duality gap.	
	\section{Scalarization}\noindent
\setcounter{equation}{0}	

	The scalarization method refers to transforming an vector optimization problem into a value optimization problem, this approach is a crucial method for solving vector optimization problems. In this section, we formulate a scalar minimization problem that can find at least one solution to the GOP \eqref{question}.
	
	Before proceeding to the scalarisation, let us recall the following definitions.
	
	\begin{Definition}\label{de7.1}(\cite{GF3})
		Let $Z$ be a nonempty set and $A$ be a convex cone in $R^n$ with $\text{int} A\neq\phi$. A function $F: Z\to{R}^n$ is said to be $A$-subconvexlike iff there exists $ a\in\text{int}A,$ such that for any $\epsilon>0$ and $\alpha\in[0,1],$
		$$(1-\alpha)F(Z)+\alpha F(Z)+\epsilon\alpha\subseteq F(Z)+A.$$
		
		If the condition $\text{``} $ there exists a $\in\text{int}A \text{''}$ is replaced by $\text{``}$ for any $a \in{A} \text{''}$, then the
		function $F: Z\to{R}^n$ is $A$-convexlike functions.
	\end{Definition}
	
	\begin{Definition}\label{de7.2}
	Let $f: M\to M$ be a given function, for any $x^{\prime}, x^{\prime\prime}\in M$ and $\alpha\in[0,1]$, a function $F(x)=\exp_{f(y)}^{-1}f(x): M\to T_{f(y)}M $ is said to be a $C_{f(y)}$-function iff it satisfies
	\begin{equation}\label{7.1}
		\left((1-\alpha)F(x^{\prime})+\alpha F(x^{\prime\prime})\right)-\exp_{f(y)}^{-1}\left((1-\alpha)f(x^{\prime})+\alpha f(x^{\prime\prime})\right)\in C_{f(y)}.
	\end{equation}
Specially, when $C_{f(y)}\subseteq T_{f(y)}M$, $F(X)$ is called $C_{f(y)}$-convex function.
\end{Definition}
	
	Now,  for any ${y} \in \mathrm{M}$, we consider the following sets:
	$$G(y)=\left\{x\in M\;|\;-\exp_{f(y)}^{-1}f(x)\in C_{f(y)}\right\}$$
	and
	$$G_p(y)=\{x\in M\;|\;\langle p,\exp_{f(y)}^{-1}f(y)\rangle\geq\langle p,\exp_{f(y)}^{-1}f(x)\rangle\}.$$
	
In the following discussion, we consider $p$ as a fixed value rather than a parameter. It is important to note that $y$ is now a parameter.

	\begin{Proposition}\label{pr7.1}
		If $F(x)=\exp_{f(y)}^{-1}$ is a $C_{f(y)}$-function, then for any ${y} \in \mathrm{M}$, the set $G(y)$ is convex .
	\end{Proposition}
	
	\begin{proof}
	Since $x^{\prime},x^{\prime\prime}\in G(y)$, then there exists  $c_{f(y)}^{\prime},\;c_{f(y)}^{\prime\prime}\in C_{f(y)},$ such that
		$$\exp_{f(y)}^{-1}f(x^{\prime})=0-c_{f(y)}^{\prime}=\exp_{f(y)}^{-1}f(y)- c_{f(y)}^{\prime},$$ $$\exp_{f(y)}^{-1}f(x^{\prime\prime})=0-c_{f(y)}^{\prime\prime}=\exp_{f(y)}^{-1}f(y)-c_{f(y)}^{\prime\prime}.$$
		Because $C_{f(y)}$ is convex for any $\alpha\in[0,1]$, we have
		$$\check{c}_{f(y)}=(1-\alpha){c_{f(y)}}^{\prime}+\alpha{c_{f(y)}}^{\prime\prime}\in C_{f(y)}.$$
		Therefore, for any $\alpha\in[0,1]$, we can conclude that
		\begin{equation}\label{7.2}
			(1-\alpha)\exp_{f(y)}^{-1}f(x^{\prime})+\alpha \exp_{f(y)}^{-1}f(x^{\prime\prime})=\exp_{f(y)}^{-1} f(y)-\check{c}_{f(y)}.
		\end{equation}
		By \eqref{7.1}, there exists $\hat{c}_{f(y)}\in C_{f(y)},$
		such that
		$$\begin{aligned}
			\exp_{f(y)}^{-1} f\big((1-\alpha)x' + \alpha x''\big) 
			&= (1-\alpha) \exp_{f(y)}^{-1} f(x') + \alpha \exp_{f(y)}^{-1} f(x'') - \hat{c}_{f(y)} \\
			&= \exp_{f(y)}^{-1} f(y) - \check{c}_{f(y)} - \hat{c}_{f(y)} \\
			&= \exp_{f(y)}^{-1} f(y) - c_{f(y)},
		\end{aligned}$$
		since $M$ is convex, then $(1-\alpha)x^{\prime}+\alpha x^{\prime\prime}\in M.$  Additionally,  $c_{f(y)}=\check{c}_{f(y)}-\hat{c}_{f(y)}\in C_{f(y)}$, since ${C}_{f(y)}$ is a convex cone, and the penultimate equation follows from  \eqref{7.2}. For any  $x^\prime,x^{\prime\prime}\in G(y)$, it follows that
		$$(1-\alpha)x^\prime+\alpha x^{\prime\prime}\in G(y).$$
	\end{proof}

	Let $y\in M$ be a parameter, $p\in{C_{f(y)}}^*$ a given point and $$K(y)=\left\{x\in M|\exp_{f(y)}^{-1}f(x)\notin C_{f(y)},g(x)\geq0\right\},$$ the (scalar) quasi-minimum problem is to find $x\in M,$ such that:
	\begin{equation}\label{7.3}
		\begin{split}
			\min & \langle p, \exp_{f(y)}^{-1} f(x) \rangle, \\
			\text{s.t.} & \quad x \in K \cap G(y).
		\end{split}
	\end{equation}

	\begin{Definition}\label{de7.3}
		The feasible point $x^{\prime}\in K$ is called an efficient solution of the quasi-minimum problem \eqref{7.3} iff
		$$\langle p,\exp_{f(y)}^{-1}f(x)\rangle\nless\langle p,\exp_{f(y)}^{-1}f(x^{\prime})\rangle,\;\;\forall x\in K\cap G(y).$$
		For any $x\in K\cap G(y)$, the inequality here is equivalent to
		\begin{equation}\label{7.4}
			\langle p,\exp_{f(x)}^{-1}f(x^{\prime})\rangle\nless0.
		\end{equation}
	\end{Definition}
	
	\begin{Proposition}\label{pr7.2}
		If $\exp_{f(y)}^{-1}$ is a $C_{f(y)}-function$, $Q(x)=\langle p,\exp_{f(y)}^{-1}(x)\rangle $, $g$ is a concave, and $p\in{C_{f(y)}}^*$, then the set  $\{x\in M|\langle p,\exp_{f(x)}^{-1}f(x^{\prime})\rangle\nless0\}$ is convex.
	\end{Proposition}
	
	\begin{proof}
		It must be shown that $Q(x)=\langle p,\exp_{f(y)}^{-1}(x)\rangle $ and $K\cap G(y)$ are convex. Since ${C_{f(y)}}^*$ and \eqref{7.1}, we have for any $ x^{\prime},x^{\prime\prime}\in M,$ and for any $\alpha\in[0,1]$, there are
		$$\langle p,(1-\alpha)\exp_{f(y)}^{-1}f(x^{\prime})+\alpha \exp_{f(y)}^{-1}f(x^{\prime\prime})-\exp_{f(y)}^{-1}((1-\alpha)f(x^{\prime})+\alpha f(x^{\prime\prime}))\rangle\geq0,$$
		or
		$$\langle p,\exp_{f(y)}^{-1}((1-\alpha)f(x^{\prime})+\alpha f(x^{\prime\prime}))\rangle\leq(1-\alpha)\langle p,\exp_{f(y)}^{-1}f(x^{\prime})\rangle+\alpha\langle p,\exp_{f(y)}^{-1}f(x^{\prime\prime})\rangle.$$
	This implies that the convexity of $\langle p,\exp_{f(y)}^{-1}(x)\rangle $, $M$ convexity and the concavity of $g$ can be deduced from the convexity of $K$. Furthermore, by  Proposition \ref{pr7.1}, we can establish the convexity of convexity of $G(y)$ and hence that of $K\cap G(y)$. 
	\end{proof}
	
	\begin{Proposition}\label{pr7.3}
		For any $y\in M$, if $p\in{C_{f(y)}}^*$,  then
		
	\begin{equation}\label{7.5}
		\left\{
		\begin{array}{ll}
			G(y) \subseteq G_p(y) & \text{(1)} \\
			y \in G(y) \cap G_p(y) & \text{(2)}
		\end{array}
		\right.
	\end{equation}

	\end{Proposition}
	
	\begin{proof}
		Since $x\in G(y)$, we have exists $c_{f(y)}\in C_{f(y)}$ such that $$\exp_{f(y)}^{-1}f(x)=\exp_{f(y)}^{-1}f(y)-c_{f(y)}.$$ According to this equation, considering $p\in{C_{f(y)}}^*$ and $c_{f(y)}\in C_{f(y)}$, it can be deduced that $\langle p,c_{f(y)}\rangle\geq0$, for any $y\in M$, we have 
		$$\langle p,\exp_{f(y)}^{-1}f(x)\rangle=\langle p,\exp_{f(y)}^{-1}f(y)\rangle-\langle p,c_{f(y)}\rangle\leq\langle p,\exp_{f(y)}^{-1}f(y)\rangle.$$
		The first equation of \eqref{7.5} has been proven. $0\in C_{f(y)}$ is equivalent to $y\in G(y)$. Additionally,  $y\in G_p(y)$ is trivial. Therefore, the second equation of \eqref{7.5} holds.
	\end{proof}
	
We will now present several properties that could be useful in defining a method for finding one or all the solutions to \eqref{question} by solving
	\eqref{7.3} .
	\begin{Proposition}\label{pr7.4}
		Let $p\in \text{int}C_{f(y)}^*$ be fixed, then \eqref{EQ4} impossible, and hence $y$ is an efficient solution of \eqref{question} iff for any $x\in M$
		
		\begin{equation}\label{7.6}
			\langle p,\exp_{f(y)}^{-1}f(x)\rangle>0,\;\exp_{f(y)}^{-1}f(x)\in C_{f(y)},\;g(x)\geq0,
		\end{equation}
		is impossible. Moreover, the impossibility of \eqref{7.6} is a necessary and sufficient condition for $y$ to be a (scalar) minimum point of \eqref{7.3}. 
	\end{Proposition}
	
	\begin{proof}
		Since $\langle p,\exp_{f(y)}^{-1}f(x)\rangle>0$, we have $\exp_{f(y)}^{-1}f(x)\neq0$. Consequently, the possibility of \eqref{7.6} implies that of \eqref{EQ4}. The first equation of \eqref{EQ4}, combined with $p\in \text{int}{C_{f(y)}}^*$, implies the first equation of \eqref{7.6}. Therefore,  possibility of \eqref{EQ4} implies that of \eqref{7.6}.
		
		By replacing the first equation of \eqref{7.6} equivalently with $\langle p,\exp_{f(y)}^{-1}f(x)\rangle<\langle p,\exp_{f(y)}^{-1}f(y)\rangle $, we have the second part of the statement immediately.
	\end{proof}
	
	\begin{Proposition}\label{pr7.5}
		For any $y^0\in M,$
		\begin{equation}\label{7.7}
			x^0\in G(y^0)\Rightarrow G(x^0)\subseteq G(y^0).
		\end{equation}
	\end{Proposition}
	
	\begin{proof}
		Since $x^0\in G(y^0)$, we have exists $\breve{c}_{f(y)}\in C_{f(y)}$ such that  $$\exp_{f(y)}^{-1}f(x^0)=\exp_{f(y)}^{-1}f(y^0)-\check{c}_{f(y)}.$$
		
		Since $\hat{x}\in G(y^0)$, we have $$\exists\dot{c}_{f(y)}\in C_{f(y)}$$ such that  $\exp_{f(y)}^{-1}f(\hat{x})=\exp_{f(y)}^{-1}f(x^0)-\dot{c}_{f(y)}.$
		
		By summing the two equalities, we obtain $$\exp_{f(y)}^{-1}f(\hat{x}) =\exp_{f(y)}^{-1}f(y^0) - c_{f(y)},$$ where $c_{f(y)}= \check{c}_{f(y)} + \dot{c}_{f(y)} \in C_{f(y)}$, since $C$ is a convex cone. 
		
		It follows that $\hat{x} \in G(y^0)$ and hence $G(x^0) \subseteq G(y^0)$. 
	\end{proof}
	
	\begin{Proposition}\label{pr7.6}
		If $x^{0}$ is an efficient solution of \eqref{7.3} at $y=y^0$, then $x^{0}$ is an efficient solution of \eqref{7.3} also at $y=x^0$.
	\end{Proposition}
	
	\begin{proof}
		Proof by contradiction, suppose that $x^0$ be not an efficient solution of \eqref{7.3} at $y=x^0$. Then
		\begin{equation}\label{7.8}
			\exists\hat{x}\in K\cap G(x^0)\quad\mathrm{~s.t.~}\langle p,\exp_{f(y)}^{-1}f(x^0)\rangle>\langle p,\exp_{f(y)}^{-1}f(\hat{x})\rangle.
		\end{equation}
		
		Since Proposition \ref{pr7.5}, if $x^0 \in G(y^0)$, then $G(x^0) \subseteq G(y^0)$, this inclusion relation and \eqref{7.8} imply 
		$$\hat{x}\in K\cap G(x^0) \;\;\text{and}\;\; \langle p,\exp_{f(y)}^{-1}f(x^0)\rangle>\langle p,\exp_{f(y)}^{-1}f(\hat{x})\rangle,$$ 
		which contradiction the assumption. 
	\end{proof}
		
	Proposition \ref{pr7.6} presents a method for solving the GOP of  \eqref{question}. We first choose any $p\in \text{int}{C_{f(y)}}^*$, with $p$ held constant in the subsequent process. Then, we choose any $y^0\in K$ and solve problem \eqref{7.3} at $y = y^0$ (scalar problem). If there exists a solution $x^0$ (if there is one), then by Proposition \ref{pr7.6}, $x^0$ is an efficient solution to  \eqref{question}. If we want to find all solutions of  \eqref{question}, which happens, for example, when optimising the set of efficient solutions of  \eqref{question} for a given function, we have to consider  \eqref{7.3} as a parametric problem with respect to $y$. Proposition \ref{pr7.4} and \eqref{7.6} ensure that we have access to all solutions of  \eqref{question}. It is worth noting that this scalarisation method does not require any assumptions on \eqref{question}.
	
	\section{Conclusions}\noindent
	\setcounter{equation}{0}
	This paper explores GOPs on a local sphere using the cone order relation in tangent spaces. The analysis primarily employs ISA, leveraging two classes of weak separation functions to characterize the sufficient and necessary optimality conditions for GOP. Last, we establish a scalar minimization problem. It is believed that there are a large number of interesting problems worthy of study in the scope of ISA on local sphere. For instance, since $C_{x}$ is a cone in the tangent space, the set defined as $\widetilde{C}=\{\exp_xc_x\;|\;\forall c_x\in C_x\}$ warrants thorough exploration. Additionally, the solutions of variational inequalities on the local sphere constitute an important topic. We foresee further progress in this topic in the nearby future.
	\begin{acknowledgement}
		The research presented in this paper has been supported by the High Performance Computing Center, Southwest Petroleum University. This work was supported by the National Natural Science Foundation of China (No: 11901484, 12171377) and Natural Science Starting Project of SWPU (No. 2023QHZ007).
	\end{acknowledgement}

\end{document}